\documentclass[12pt]{amsart}

\usepackage{graphicx}
\usepackage{amssymb}
\usepackage{float}
\usepackage{fullpage}

\usepackage{comment}

\newtheorem{thm}{Theorem}

\newtheorem{lem}{Lemma}

\newtheorem{theorem}{Theorem}[section]
\newtheorem{lemma}[theorem]{Lemma}
\newtheorem{corollary}[theorem]{Corollary}
\newtheorem{proposition}[theorem]{Proposition}

\theoremstyle{definition}

\theoremstyle{remark}
\newtheorem{remark}[theorem]{Remark}

\numberwithin{equation}{section}

\newcommand{\R}{{\mathbb R}}
\newcommand{\Q}{{\mathbb Q}}
\newcommand{\Z}{{\mathbb Z}}
\newcommand{\N}{{\mathbb N}}

\begin{document}

\title[k-primitive sets]{On the critical exponent for $k$-primitive sets}

\author{Tsz Ho Chan}
\address{Department of Mathematics, Kennesaw State University, Kennesaw, GA 30144}
\email{thchan6174@gmail.com}

\author{Jared Duker Lichtman}
\address{Mathematical Institute, University of Oxford, Oxford, OX2 6GG, UK}
\email{jared.d.lichtman@gmail.com}

\author{Carl Pomerance}
\address{Department of Mathematics, Dartmouth College, Hanover, NH 03755}
\email{carl.pomerance@dartmouth.edu}


\date{December 3, 2020.}



\begin{abstract}
A set of positive integers is primitive (or 1-primitive) if no member divides another. Erd\H{o}s proved in 1935 that the weighted sum $\sum1/(n \log n)$ for $n$ ranging over a primitive set $A$ is universally bounded over all choices for $A$. In 1988 he asked if this universal bound is attained by the set of prime numbers. One source of difficulty in this conjecture is that $\sum n^{-\lambda}$ over a primitive set is maximized by the primes if and only if $\lambda$ is at least the critical exponent $\tau_1 \approx 1.14$.

A set is $k$-primitive if no member divides any product of up to $k$ other distinct members. One may similarly consider the critical exponent $\tau_k$ for which the primes are maximal among $k$-primitive sets.
In recent work the authors showed that $\tau_2 < 0.8$, which directly implies the Erd\H{o}s conjecture for 2-primitive sets. In this article we study the limiting behavior of the critical exponent, proving that $\tau_k$ tends to zero as $k\to\infty$.
\end{abstract}

\maketitle

\section{Introduction}

\bigskip

A set $A\subset \Z_{>1}$ is {\it primitive} if no member of $A$ divides another. Erd\H{o}s \cite{E35} showed that for any primitive set $A$,
\begin{align*}
\sum_{n \in A} \frac{1}{n \log n} \ < \ \infty.
\end{align*}
In fact, his proof bounded the sum uniformly over all primitive sets $A$. Further, in 1988 he asked if the maximizer is the set of primes $A = \mathbb{P}$.  This is now referred to as the Erd\H os conjecture for primitive sets:
\begin{align}\label{conj:Erdos}
\hbox{\em For primitive $A$, we have }~\sum_{n \in A} \frac{1}{n \log n} \ \le \  \sum_{p \in \mathbb{P}} \frac{1}{p \log p} \ = \ 1.6366\cdots,
\end{align}
The current record bound is $\sum_{n \in A}1/(n\log n) < e^\gamma=1.781\cdots$ due to the second and third authors \cite{LP}. Here $\gamma$ is the Euler--Mascheroni constant.

A potential approach towards the Erd\H os conjecture is via integration. Namely, we have
\begin{align*}
\sum_{n \in A} \frac{1}{n \log n} \ = \ \int_1^\infty \bigg(\sum_{n \in A} \frac{1}{n^{\lambda}}\bigg)\;d\lambda,
\end{align*}
and one might hope the integrand above is dominated by $\sum_p p^{-\lambda}$ for all $\lambda>1$. Note by a simple argument (see Lemma \ref{lem:minexp}), if this inequality holds for an exponent $\lambda$, then it will continue to hold for all larger exponents $\lambda' > \lambda$.

However, the primes are not maximal among primitive sets with respect to logarithmic density (i.e., $\lambda=1$). Indeed, by Erd\H{o}s \cite{E48} and Erd\H{o}s, S\'{a}rk\"{o}zy, and Szemer\'{e}di \cite{ESS},
\begin{align*}
\sup_{\text{primitive }A} \ \mathop{\sum_{n \in A}}_{n \le x} \frac{1}{n} \ = \ \Bigl( \frac{1}{\sqrt{2 \pi}} + o(1) \Bigr) \frac{\log x}{\sqrt{\log \log x}},
\end{align*}
where the maximizer is the set of positive integers with $\lfloor \log\log x\rfloor$ prime factors (with multiplicity).
By contrast, the primes satisfy
\[
\sum_{p \le x} \frac{1}{p} = \log \log x + O(1).
\]

Later, Banks and Martin \cite{BM} obtained the full characterization that 
\begin{align}\label{eq:tau1}
\sum_{\substack{n \in A\\n\le x}}n^{-\lambda} \ \le \ \sum_{p\le x} p^{-\lambda},
\end{align}
for all primitive $A$, $x>1$, if and only if $\lambda\ge \tau_1:=1.1403\ldots$, where $\tau=\tau_1$ is the unique real solution to the equation
\begin{align}\label{eq:taudef}
\sum_{p\in\mathbb P} p^{-\tau} \ = \ 1  \ + \ \Big(1-\sum_{p\in\mathbb P} p^{-2\tau}\Big)^{1/2}.
\end{align}
As such we call $\tau_1$ the {\it critical exponent} for primitive sets.

One may define a hierarchy of primitivity as follows. A 1-primitive set is primitive, and inductively for $k>1$, a $(k-1)$-primitive set is $k$-primitive if no member divides the product of $k$ distinct other members. That is, a set $A\subset \Z_{>1}$
is {\it $k$-primitive} if no member of $A$ divides any product of $j$ distinct other members, for any $1\le j\le k$.\footnote{In \cite{CGS}, $A$ is called $k$-primitive if no member of $A$ divides any product of $k$ distinct other members. These definitions only differ when $|A|\le k$, and do not affect the critical exponent $\tau_k$, by Lemma \ref{Ysmall} below.}
Note that if \eqref{eq:tau1} holds for all $\lambda>\tau$, then it holds for $\lambda=\tau$.  Thus,
one may similarly consider the critical exponent $\tau_k$ for which \eqref{eq:tau1} holds for all $k$-primitive sets if and only if $\lambda\ge \tau_k$. Note that $\tau_j\ge \tau_k$ for $1\le j\le k$. 

Recently, the authors \cite{CLP} proved $\tau_2\le 0.7983$. In particular $\tau_2<1$, thereby establishing the Erd\H{o}s conjecture in the case of 2-primitive sets.

\begin{thm}[\cite{CLP}]
\label{CLPthm}
For $\lambda\ge 0.7983$, we have
\begin{align*}
\sum_{\substack{n \in A\\n\le x}}n^{-\lambda} \ \le \ \sum_{p\le x} p^{-\lambda}
\end{align*}
for all $2$-primitive sets $A$ and $x\ge2$. In particular, any $2$-primitive set $A$ satisfies
\begin{align*}
\sum_{n \in A} \frac{1}{n\log n} \ \le \ \sum_p \frac{1}{p\log p}.
\end{align*}
\end{thm}

In 1938, Erd\H{o}s \cite{E38} first studied the maximal cardinality of 2-primitive sets (i.e., $\lambda=0$). 
The first author together with Gy\H{o}ri and S\'{a}rk\"{o}zy \cite{CGS} extended it to all $k\ge2$, also see \cite{C11} and \cite{PS}. Namely, there is an absolute constant $c>0$ such that
\begin{align}\label{eq:CGS}
\frac{1}{8k^2} \frac{x^{\frac{2}{k+1}}}{(\log x)^2} \ \le \ \sup_{k\text{-primitive }A} \ \mathop{\sum_{n \in A}}_{n \le x} 1 \ - \ \sum_{p\le x}1 \ \le \ ck^2 \frac{x^{\frac{2}{k+1}}}{(\log x)^2},
\end{align}
for $x$ sufficiently large. Here the lower bound is attained by some set $A''$ consisting of the primes in $(x^{1/(k+1)}, x]$ and a size $x^{2/(k+1)}/8 (k\log x)^2$ subset of products of $k+1$ primes in $(1, x^{1/(k+1)}]$. In particular, the lower bound in \eqref{eq:CGS} implies
\[
\sum_{\substack{n \in A''\\n\le x}}n^{-\lambda} \ge \sum_{x^{1/(k+1)} < p \le x} p^{-\lambda} + \frac{1}{x^\lambda} \frac{x^{2/(k+1)}}{8 (k\log x)^2} > \sum_{x^{1/(k+1)} < p \le x} p^{-\lambda} + \sum_{p \le x^{1/(k+1)}} p^{-\lambda},
\]
when $\lambda < 1/k$ and $x$ is sufficiently large. Hence we quickly deduce $\tau_k \ge 1/k$.

Thus combining with Theorem \ref{CLPthm}, the critical exponent for 2-primitive sets lies in the interval
\begin{align}
\tau_2 \ \in \ [0.5,\,0.7983].
\end{align}
It is an open question to determine the exact value of $\tau_2$, and perhaps characterize $\tau_2$ as a solution to some functional equation, as with \eqref{eq:tau1} for $\tau_1$.

In light of this, it is natural to ask about the behavior of the decreasing sequence 
 $\tau_1 \ge \tau_2 \ge \tau_3 \ge \cdots$, in particular whether $\tau_k$ tends to zero as $k\to\infty$. The main result of this article is to answer in the affirmative, with the following quantitative result. 

\begin{thm}\label{thm:main}
Let $p_k$ denote the $k$th prime number.
For any $k \ge 1$ and $\lambda \ge 1.5/\log p_k$, we have
\begin{align}\label{eq:mainthm}
\sum_{\substack{n\le x\\n \in A}} n^{-\lambda} \ \le \ \sum_{p\le x} p^{-\lambda}
\end{align}
for all $k$-primitive sets $A$ and $x\ge2$.
\end{thm}

Thus, for $k\ge1$,
\begin{align}
\frac{1}{k} \ \le \ \tau_k \ \le \ \frac{1.5}{\log p_k}.
\end{align}
Clearly the upper and lower bounds differ substantially, and it remains an unsolved problem to narrow this gap.

\subsection{Generalizations}

Upon closer inspection of the proofs in \cite{C11}, \cite{CGS}, we observe the lower bound for \eqref{eq:CGS} holds under a stronger notion of $k$-primitivity, namely, one forbids a member from dividing the product of $k$ other members, {\it not necessarily distinct}. Similarly, the upper bound in \eqref{eq:CGS} holds even if one relaxes to only forbid a member from dividing the {\it least common multiple} (lcm) of $k$ other members.

Hence this naturally suggests the following generalizations. We say a set $A\subset \Z_{>1}$ is ``strongly $k$-primitive" if no member divides the product of $k$ other members which are not necessarily distinct.  Any strongly $k$-primitive set is $k$-primitive, but not vice versa.  For example, $A=\{4,5,6\}$ is 2-primitive but not strongly 2-primitive.
In the other direction, we say a set $A\subset \Z_{>1}$ is ``lcm $k$-primitive" if no member divides the lcm of $k$ other members.  Here, every $k$-primitive set is lcm $k$-primitive,
but not vice versa.  An example is $A=\{4,6,10\}$ which is lcm 2-primitive, but not 2-primitive.

One can ask for critical exponents in the strong case and in the lcm case.  Denote the former
by $\tau^{(\rm{s})}_k$ and the latter by $\tau^{(\rm{lcm})}_k$.  By the above comments, for each $k\ge2$ we have
\begin{align*}
\tfrac{1}{k} \ \le \ \tau^{(\rm{s})}_k \ \le \ \tau_k \ \le \ \tau^{(\rm{lcm})}_k.
\end{align*}

From these definitions, two natural questions arise:  Is there a better upper bound for $\tau^{(\rm{s})}_k$ than that afforded by Theorem
\ref{thm:main}?  Is there an upper bound for $\tau^{(\rm{lcm})}_k$ that is $o(1)$ as $k\to\infty$?

We make progress on these two questions by proving the following two theorems.

\begin{thm}\label{thm:lcm}
For any $k \ge 1$ $\tau_k^{(\rm lcm)}\le 1.7/\log p_k$.
In addition, $\tau_2^{(\rm lcm)}\le 1$, so the Erd\H os conjecture is true for {\rm lcm} {\rm 2}-primitive sets.
\end{thm}

For the $\tau^{(\rm{s})}_k$ case we 
prove a considerably stronger inequality.
\begin{thm}
\label{indistinct}
For $k \ge 2$ we have $\tau_k^{(\rm s)}\le (3\log k)/k$.
\end{thm}
Thus,
\[
\frac1k\le\tau_k^{(\rm s)}\le\frac{3\log k}k
\]
for all $k\ge2$.  It would be nice to so sharpen the inequalities for $\tau_k$ and $\tau_k^{(\rm lcm)}$.




\section{Preliminary lemmas}

\begin{lem}\label{lem:minexp}
Take sets $A,B\subset \R_{>1}$. Suppose $\lambda\ge 0$ satisfies $I_\lambda(x)\ge0$ for all $x>1$, where
\begin{align*}
I_\lambda(x) := \sum_{\substack{a\in A\\ a\le x}} a^{-\lambda} \ - \ \sum_{\substack{b\in B\\ b\le x}} b^{-\lambda}.
\end{align*}
Then $I_{\lambda'}(x)\ge0$ for all $\lambda'\ge \lambda$, $x>1$.
\end{lem}
\begin{proof}
By partial summation,
\begin{align*}
I_{\lambda'}(x) \ = \ x^{\lambda - \lambda'}I_{\lambda}(x) \ + \ (\lambda'-\lambda)\int_1^x u^{\lambda-\lambda'-1} I_{\lambda}(u) \;du.
\end{align*}
Hence if $I_{\lambda_k}(x)\ge0$ for all $x>1$, it then follows $I_{\lambda'}(x)\ge0$ for all $\lambda'\ge \lambda$ as claimed.
\end{proof}

\begin{lem} \label{primeineq}
Let 
$$
\lambda_1=1.2,\quad \lambda_2=0.8, \quad\hbox{ and }~\lambda_k = 2.625 \prod_{i=1}^k\Big(1-\frac1{p_i}\Big)~ \hbox{ for }~k\ge3.
$$
Then
\begin{align*}
\lambda_k>\frac{1.45}{\log p_k}~\hbox{ for }~k\ge 62,\qquad
\lambda_k<\frac{1.5}{\log p_k}~\hbox{ for }~k\ge1.
\end{align*}
In addition, let
$$
\mu_1=8/7 \quad\hbox{ and }~\mu_k = 3 \prod_{i=1}^k\Big(1-\frac1{p_i}\Big)~ \hbox{ for }~k\ge2.
$$
Then
\begin{align*}
\mu_k>\frac{1.65}{\log p_k}~\hbox{ for }~k\ge 47,\qquad
\mu_k<\frac{1.7}{\log p_k}~\hbox{ for }~k\ge1.
\end{align*}
\end{lem}
\begin{proof}
One can verify the lemma for $p_k \le 2{,}000$ by direct computation. 
For larger $p_k$ we use (3.25) of Rosser and Schoenfeld \cite{RS} with the Euler--Mascheroni constant $\gamma = 0.57721...$, getting
\[
\lambda_k \ge \frac{2.625 e^{-\gamma}}{\log p_{k}} \Bigl(1 - \frac{1}{2\log^2 p_{k}}\Bigr) \ge\frac{2.625 e^{-0.57722}}{\log p_k} \Bigl(1 - \frac{1}{2\log^2 2{,}000} \Bigr) \ge \frac{1.45}{\log p_k},
\]
which gives the lower bound for $\lambda_k$.  The lower bound for $\mu_k$ follows in the same way.
For the upper bound, 
by (3.26) of
Rosser and Schoenfeld \cite{RS} we have
\[
\lambda_k<\frac{2.625e^{-.57721}}{\log p_k}\Big(1+\frac1{2\log^22{,}000}\Big)
<\frac{1.5}{\log p_k}.
\]
Again, the upper bound for $\mu_k$ follows in the same way.
This completes the proof.
\end{proof}

\begin{lem} \label{primeineq2}
For $0 < \lambda < 1$ and $x \ge 41$,
\[
x^{1 - \lambda} \Bigl(1 - \frac{1}{\log x}\Bigr) \le \sum_{p \le x} \frac{\log p}{p^\lambda} \le \frac{1.01624}{1 - \lambda} x^{1 - \lambda}.
\]
\end{lem}
\begin{proof}
By partial summation,
\[
\sum_{p \le x} \frac{\log p}{p^\lambda} = \int_{2^-}^{x} \frac{d \theta(u)}{u^\lambda} = \frac{\theta(x)}{x^\lambda} + \lambda \int_{2}^{x} \frac{\theta(u)}{u^{\lambda + 1}} du
\]
where $\theta(x) = \sum_{p \le x} \log p$. The lemma follows from (3.16) and (3.32) in Rosser and Schoenfeld
\[
x \Bigl(1 - \frac{1}{\log x}\Bigr) < \theta(x) \; \text{ for } \; x \ge 41
\]
and
\[
\theta(x) < 1.01624 x \; \text{ for } \; x > 0.
\]
\end{proof}

For a set $A$ of integers, let ${\mathcal P}(A)$ denote the set of primes that divide some member of $A$.
\begin{lem}
\label{Ysmall}
Let $A$ be an {\rm lcm} $k$-primitive set with $k\ge2$.
 If $|\mathcal P(A)|\le k$, then $|A|\le|\mathcal P(A)|$ and for all $\lambda\ge 0$,
\begin{align*}
\sum_{n \in A} n^{-\lambda}
\ \le \ \sum_{p\in \mathcal P(A)}p^{-\lambda}.
\end{align*}
Also, if $k<|\mathcal P(A)|<2k$, then $|A|\le |\mathcal P(A)|+1$.
\end{lem}
\begin{proof}
Let $v_p(n)$ denote the exponent on $p$ in the prime factorization of $n$, so that $p^{v_p(n)}\parallel n$.
For each $p\in \mathcal P(A)$ let $n_p$ be the element $n\in A$ with $v_p(n)$ maximal (breaking ties arbitrarily), and let $A^*=\{n_p:p\in\mathcal P(A)\}$.  Thus
$|A^*|\le |\mathcal P(A)|$.

Suppose $|\mathcal P(A)|\le k$.  Then any $n\in A\setminus A^*$ would satisfy $n \mid {\rm lcm}(A^*)$, contradicting $A$ as lcm $k$-primitive.  Thus,
$A^*=A$ and $|A|\le|\mathcal P(A)|\le k$.
Next, $|A|\le k$ implies each $n\in A$ has $n\nmid {\rm lcm}(A\setminus\{n\})$. Thus, each $n\in A$ has a prime factor $p$ with $v_{p}(n) > v_{p}(m)$ for all $m\in A\setminus\{n\}$, so the map, call it $f$, where $f(n)=p\mid n$ is injective on $A$. Hence we conclude
\begin{align*}
\sum_{n \in A} n^{-\lambda} 
\ \le \ \sum_{n\in A}f(n)^{-\lambda} 
\ \le \ \sum_{p\in \mathcal P(A)}p^{-\lambda}.
\end{align*}

Also, suppose $N=|\mathcal P(A)|$, $k<N<2k$, and there exist distinct $n,n'\in A\setminus A^*$. Without loss, the subset $P=\{p\in\mathcal P(A) : v_p(n) \ge v_p(n')\}$ contains at least half of the primes in $\mathcal P(A)$, i.e., $|P|\ge \lceil \frac{N}{2}\rceil$. Hence
\begin{align*}
n' \ \mid \ {\rm lcm}(\{n_p: p\notin P\}\cup \{n\}),
\end{align*}
which is an lcm of $1+N-\lceil N/2\rceil$ elements.   It is easy to see this number is $\le k$, thus contradicting $A$ as lcm $k$-primitive. This implies $|A|\le N+1$.
\end{proof}
\section{Theorem for $k$-primitive sets}
\label{sec:main}

In this section we prove Theorem  \ref{thm:main}.  
Recall the numbers $\lambda_k$ in Lemma \ref{primeineq}.  By that lemma it suffices to prove the following theorem.
\begin{thm} 
\label{thm:actuallymain}
Let $A$ be a $k$-primitive set.  For each $k \ge 1$ 
we have
\begin{align}
\label{mainineq}
\sum_{\substack{a \in A\\a\le x}} a^{-\lambda_k} \ \le \ \sum_{\substack{p \in \mathcal{P}(A)\\p\le x}} p^{-\lambda_k}.
\end{align}
for any $x>1$.
\end{thm}

Since $\lambda_1=1.2, \lambda_2=0.8$, the theorem holds for $k=1,2$, so we may assume that $k\ge3$ and that the theorem holds for $(k-1)$-primitive sets.

 We partition $A$ into primes $S$ and composites $T$. Note by primitivity, the primes in $S$ and $\mathcal P(T)$ are disjoint. We thus may cancel the contribution of $p\in S$ from both sides of \eqref{mainineq} and so reduce
Theorem~\ref{thm:actuallymain} to the case $A=T$ where every member is composite.

We may assume that
\begin{align}
\label{eq:ptp2}
\sum_{t\in T_p}t^{-\lambda} \ > \ p^{-\lambda} \quad \text{for all}\quad p \in \mathcal{P}(T),
\end{align}
since if this fails for some $p$, the theorem
for $T\setminus T_p$ implies the theorem for $T$.  
An immediate consequence is that
\begin{align}
\label{two}
|T_p| \ge 2 \quad \text{for all}\quad p \in \mathcal{P}(T).
\end{align}

Further, it suffices to assume that $\mathcal{P}(T)$ consists of an initial list of primes, say
\begin{align}
 \mathcal{P}(T) = \mathbb{P} \cap (1, Y] \quad \text{for some}\quad Y\ge2.
\end{align}
Indeed, if not, suppose $q$ is the smallest prime outside $\mathcal{P}(T)$, and let $p\in\mathcal{P}(T)$ be the smallest prime with $p>q$. Then by \eqref{eq:ptp2},
\begin{align*}
0 \ < \ (p/q)^\lambda\bigg(\sum_{t\in T_p}t^{-\lambda} \ - \ p^{-\lambda}\bigg) \ \le \  \sum_{t'\in T'_q}(t')^{-\lambda} \ - \ q^{-\lambda},
\end{align*}
where $T'$ is the ($k$-primitive) image of $T$ under the automorphism of $\N$ induced by swapping $q\leftrightarrow p$. Hence the proof for $T$ will follow from that of $T'$.

For an integer $t>1$ let $Q(t)$ denote the
largest prime power factor of $t$, which is possibly a prime to the first power.
We first handle those $t \in T$ with $Q(t) < t^{\theta}$ for an appropriate choice of $\theta$.  

\begin{lem} \label{smooth2}
Let $k\ge2$ and let $0< \theta\le 1/k$.
Suppose $T$ is {\rm lcm} $k$-primitive with $Q(t) < t^{\theta}$ for each $t \in T$. Let $z\ge 2$, and let $N(z)$ be the number of members of $T$ up to $z$. Then
\begin{align*}
N(z) \ \le \ z^{\frac{1}{k} + \theta}.
\end{align*}
\end{lem}
\begin{proof}
If $t\le z^{1/k}$, let $m_1(t)=t$.  Now suppose that $t>z^{1/k}$ and decompose
$t = q_1 q_2 \cdots q_r$ into its prime powers $q_1>\cdots> q_r$.
By assumption, $q_1 < t^{\theta}$. Consider $q_1 \cdots q_j \le z^{1/k}$ with $j$ maximal. Then $m_1(t) := q_1\cdots q_{j+1}$ lies in the interval $(z^{1/k},z^{1/k + \theta}]$. In this way we may split $t$ into $l_t \le k$ pairwise coprime factors
\begin{equation} \label{tfactor}
t = q_1 q_2 \cdots q_r = m_1(t) \cdots m_{l_t}(t)
\end{equation}
with each $m_i(t) \le z^{1/k+\theta}$.

Now observe each $t\in T$ has some factor $m_i(t)$ which is distinct from all other factors $m_j(s)$, $s\in T\setminus \{t\}$. Indeed, if not, then each factor of $t$ has $m_i(t) = m_{j_i}(t_i)$ for some $t_i\in T\setminus \{t\}$ (not neccessarily distinct). And since the factors $m_i(t)$ are pairwise coprime,
\begin{align*}
t=m_1(t)\cdots m_{l_t}(t)\mid{\rm lcm}\,[t_1,\dots,t_{l_t}],
\end{align*}
contradicting $T$ as lcm $k$-primitive.

Hence we have a 1-1 map $g:T\to \N$ via $g(t) = m_i(t)$. And since $m_i(t)\le z^{\frac1k+\theta}$, we conclude $|T| = |g(T)| \le z^{\frac1k+\theta}$.
\end{proof}

We now fix a choice for $\theta=\theta_k$.
Let
\[
\theta_k=\frac1{p_k}\ \hbox{ for }\ k\ne3\ \hbox{ and }\ \theta_3=\frac1{8}.
\]
Further, let $\nu_k=1/\theta_k$, so that
\[
\nu_k=p_k\ \hbox{ for }\ k\ne3\ \hbox{ and }\ \nu_3=8.
\]
With these choices we have 
$$
\lambda_k=2.4\prod_{j\le k}(1-\theta_j).
$$ 
Note that
if $Q(t) < t^{\theta_k}$, then $t$ must have at least $\nu_k+1$ distinct prime factors.
Let $P(t)$ denote the
largest prime dividing $t$, so that 
$$
p_{\nu_k+1}\le P(t)\le Q(t)<t^{\theta_k}
~\hbox{ which implies }~t>p_{\nu_k+1}^{\nu_k}.
$$
%
Thus, with $\theta=\theta_k$, $\nu=\nu_k$,  and $\lambda>\frac1k+\theta$,
\begin{equation}
    \label{small2primenew}
    \sum_{\substack{t\in T\\Q(t)<t^{\theta}}}\frac1{t^\lambda}\ 
   ~=~\int_{p_{\nu+1}^{\nu}}^{\infty}\frac{\lambda}{z^{1+\lambda}}N(z)\,dz 
 ~\le ~ \frac{ \lambda}{\lambda - \frac{1}{k} - \theta} \,p_{\nu+1}^{-\nu (\lambda - \frac{1}{k} - \theta)},
\end{equation}
by partial summation and Lemma \ref{smooth2}.

\medskip

\begin{lem}
\label{derived2}
Let $k\ge2$ and let $T$ be an {\rm lcm} $k$-primitive set of composite numbers. Decompose $T = T'\cup T''$, where $t\in T''$ if there exists another $s\in T$ with $Q(t)\mid s$; else $t\in T'$.
Define the map $f:T\to\N$ via
\[
f(t) = \begin{cases}
Q(t) & t\in T'\\
t/Q(t) & t\in T''
\end{cases}
\]
Then $f$ is $1$ to $1$ and $f(T)$ is an {\rm lcm} $(k-1)$-primitive set.  Further, the members of $f(T')$ are pairwise coprime proper prime powers.
\end{lem}
\begin{proof}
First, the map $f$ is 1-1. Indeed, suppose $f(t)=f(t')$ for some $t,t'\in T$. If $t\in T'$ then $Q(t) \nmid t'$, in particular $f(t) = Q(t)\neq f(t')\in \{Q(t'), t'/Q(t')\}$. Similarly, if $t\in T''$ then $Q(t)\mid s$ for some $s\in T\setminus\{t\}$. Thus
$1=\gcd(Q(t),t/Q(t))=\gcd(Q(t),t'/Q(t'))$ implies
\begin{align*}
t = Q(t)\cdot\frac{t}{Q(t)} = Q(t)\cdot\frac{t'}{Q(t')}
\ \Big| \ \textrm{lcm}[s,t'].
\end{align*}
Thus lcm 2-primitivity of $T$ forces $t=t'$. Hence $f$ is indeed 1-1.

Next suppose $f(T)$ is not lcm $(k-1)$-primitive. Then there exist $t\in T$ and $t_1,..,t_{k-1}\in T\setminus\{t\}$ such that
\begin{align*}
f(t) \ \Big| \ \textrm{lcm}[f(t_1),\ldots,f(t_{k-1})].
\end{align*}

If $t\in T'$ then $f(t)=Q(t)$ is a prime power, so by above $Q(t) \mid f(t_i)$ for some index $i$. Thus $Q(t)\mid t_i\in T\setminus\{t\}$, which contradicts $t\in T'$.

 Similarly if $t\in T''$, then $Q(t)\mid s$ for some $s\in T\setminus\{t\}$, and so $1=\gcd(Q(t),t/Q(t))$ gives
\begin{align*}
t = Q(t)\cdot\frac{t}{Q(t)} = Q(t)f(t) \ \Big| \ \textrm{lcm}[s,t_1,\ldots,t_{k-1}]
\end{align*}
contradicting $T$ as lcm $k$-primitive. Hence $f(T)$ is indeed lcm $(k-1)$-primitive.
That the members of $f(T')$ are pairwise coprime follows from $f(T')$ being a primitive set of prime powers.   That the members of $f(T')$ are proper prime powers follows from the fact that if $Q(t)$ is prime, then by \eqref{two}, $T_{Q(t)}$ has at least 2 elements, and so $t\in T''$.
\end{proof}

Let $T_\theta=\{t\in T:Q(t)\ge t^\theta\}$.
We apply Lemma \ref{derived2} to $T=T_\theta$.
Thus, by the induction hypothesis on the lcm $(k-1)$-primitive set $f(T_\theta)$, for $\lambda' :=\lambda_{k-1} = \frac{\lambda_k}{1-\theta}$ we have
\begin{align*}
\sum_{t\in T_\theta}f(t)^{-\lambda'}  \ = \ 
\sum_{t \in T'} Q(t)^{-\lambda'} + \sum_{t\in T''}\big(t/Q(t)\big)^{-\lambda'} \ = \  \sum_{d\in f(T_\theta)}d^{-\lambda'} \ \le \  \sum_{p \le Y} p^{-\lambda'}.
\end{align*}

Now if $Q(t) \ge t^{\theta}$, then $t / Q(t) \le t^{(1- \theta)}$ so that $t^{-\lambda} \le (t/Q(t))^{- \lambda / (1 - \theta)}=(t/Q(t))^{-\lambda'}$. Thus by the above,
\begin{align*}
\sum_{t \in T_{\theta}} t^{-\lambda} = \sum_{t \in T'} t^{-\lambda} + \sum_{t \in T''} t^{-\lambda} & \le \sum_{t \in T'} Q(t)^{-\lambda} + \sum_{t\in T''}\big(t/Q(t)\big)^{- \lambda / (1 - \theta)} \nonumber\\
& \le \sum_{t \in T'} \big(Q(t)^{-\lambda} -Q(t)^{-\lambda'} \big) \ + \ \sum_{p \le Y} p^{-\lambda'}.
\end{align*}
Thus,
\begin{equation}
\label{large2prime}
 \sum_{t \in T_{\theta}} t^{-\lambda} -\sum_{p\le Y}p^{-\lambda}<  
  \sum_{p \le Y} \bigl((p^{- 2 \lambda} - p^{-2 \lambda'}) - (p^{- \lambda} - p^{-\lambda'}) \bigr)=:S(Y),
\end{equation}
using that $f(T')$ is a set of pairwise coprime proper prime powers and $\mathcal P(T)\subset [1,Y]$.
Note that from Lemma \ref{Ysmall} we may assume that $Y\ge p_{k}$.  
\medskip

\noindent{\em Claim 1:}
The sequence $S(p_j)$ for $j\ge k$ is decreasing, so if $S(p_k)<0$, then $S(Y)<0$ for all $Y\ge p_{k}$.  
\medskip

Indeed, the terms in $S(Y)$ are
of the form $h(y,z)=y-z-(y^2-z^2)$, where $y=p^{-\lambda'}$ and $z=p^{-\lambda}$.
Note that $h(y,z)=(y-z)(1-(y+z))$ and we have
$0<y<z$.  Further, $p^{-\lambda}\le\frac13$ for
$p\ge p_k$ and $k\ge3$, which follows from Lemma \ref{primeineq} and a short calculation.
Thus, for $p\ge p_k$, the terms in $S(Y)$ are negative, establishing Claim 1.
\medskip

\noindent{\it Claim 2:}
For $k\ge 3$ we have $S(p_k)<0$ and for $k\ge200$ we have $S(p_k)<-0.015/\log p_k$.
\medskip

We verify this directly for $3\le k\le 199$, so assume now that $k\ge200$.
Let $F(\lambda)=\sum_{p\le p_{k}}(p^{-2\lambda}-p^{-\lambda})$ so that
$S(p_{k})=F(\lambda)-F(\lambda')$.
By the mean value theorem, there exists some
$\xi\in(\lambda,\lambda')$ with
\begin{align*}
F(\lambda)-F(\lambda')&=(\lambda-\lambda')F'(\xi)=(\lambda-\lambda')\sum_{p\le p_k}\big(p^{-\xi}\log p-2p^{-2\xi}\log p\big)\\
&=-\theta\lambda'\sum_{p\le p_k}\big(p^{-\xi}-2p^{-2\xi}\big)\log p
\ < \ -\theta\lambda'\sum_{p\le p_k}\big(
p^{-\lambda'} -2p^{-2\lambda}\big)\log p.
\end{align*}
Recall that  $\theta=\theta_k$,
$\lambda=\lambda_k$, and $\lambda'=\lambda_{k-1}$.
Using Lemma  \ref{primeineq2}, we thus have
\begin{align*}
S(p_k)=F(\lambda)-F(\lambda')&<-\theta\lambda'\Big(p_k^{1-\lambda'}\big(1-\frac1{\log p_k}\big)-\frac{2.03248}{1-2\lambda}p_k^{1-2\lambda}\Big)\\
&=-\lambda'p_k^{-\lambda}\Big(p_k^{\lambda-\lambda'}\Big(1-\frac1{\log p_k}\Big)-\frac{2.03248}{1-2\lambda}p_k^{-\lambda}\Big).
\end{align*}
We use $1-2\lambda>0.587$, $p_k^{\lambda-\lambda'}>1-1/p_k$,
and $e^{-1.5}<p_k^{-\lambda}<e^{-1.45}$,
which follows from Lemma \ref{primeineq}, 
to get
\begin{align}
\label{Sineq}
    S(p_k)<
-\frac{0.015}{\log p_k},\qquad\hbox{for }~ k\ge200,
\end{align}
completing the proof of Claim 2.
\medskip

By \eqref{small2primenew} and \eqref{large2prime},
\begin{equation} \label{almostdone2}
I_\lambda  = \sum_{\substack{t \in T\\Q(t) < t^{\theta}}} t^{-\lambda} + \sum_{\substack{t \in T\\Q(t) \ge t^{\theta}}} t^{-\lambda} \ - \ \sum_{p \le Y} p^{-\lambda} 
< \ \frac{ \lambda}{\lambda - \frac{1}{k} - \theta}\, p_{\nu+1}^{-\nu(\lambda - \frac{1}{k} - \theta)} 
 + S(Y).
\end{equation}
Note though that if $Y<p_{\nu+1}$, then the first term does not appear, so Claims 1 and 2 prove that $I_\lambda<0$.  So, assume that $Y=p_{\nu+1}$ in \eqref{small2primenew}.  We check numerically that $I_\lambda<0$ for $3\le k\le 199$.

It remains to show that $I_\lambda<0$ for $k\ge200$. 
Note that if $k\ge200$, then 
\[
\lambda-\frac1k-\theta>\frac{1.4}{\log p_k},\quad \frac{\lambda}{\lambda-\frac1k-\theta}<1.05,
\]
using Lemma \ref{primeineq2}.  Thus, 
\[
\frac{ \lambda}{\lambda - \frac{1}{k} - \theta}\, p_{\nu+1}^{-\nu(\lambda - \frac{1}{k} - \theta)} ~<~1.05p_{p_k+1}^{-1.4p_{k}/\log p_k}~<~1.05p_{k}^{-1.4p_{k}/\log p_k}~=~1.05e^{-1.4p_k}.
\]
As a function of $p_k$ this expression is much smaller than $0.015/\log p_k$, in fact, this is so for $p_k\ge5$.  Thus, \eqref{Sineq} shows that $I_\lambda<0$ for $k\ge200$.
This completes the proof.

\section{Theorem for ${\rm lcm}$ $k$-primitive sets}
In this section we prove Theorem \ref{thm:lcm}.  The proof largely follows from the proof for $k$-primitive sets in the previous section.  In fact, the only difference is that we start the induction at $k=2$ rather than $k=3$.
By Lemma \ref{primeineq} it suffices to prove the following theorem.
\begin{thm}
\label{thm:actuallylcm}
Recall the numbers $\mu_k$ in Lemma \ref{primeineq}.
Let $A$ be an ${\rm lcm}$ $k$-primitive set.  For each $k \ge 1$ 
we have
\begin{align}
\label{mainineq2}
\sum_{\substack{a \in A\\a\le x}} a^{-\mu_k} \ \le \ \sum_{\substack{p \in \mathcal{P}(A)\\p\le x}} p^{-\mu_k}.
\end{align}
for any $x>1$.
\end{thm}

First note that since $\tau_1<8/7=\mu_1$, the theorem holds at $k=1$, so we may assume that $k\ge2$ and the theorem holds for lcm $(k-1)$-primitive sets.

Next note that the various reductions we made in Section \ref{sec:main} hold here, as well as Lemmas \ref{smooth2} and \ref{derived2}.
Here we have 
\[
\theta_k=1/p_k~\hbox{ for }~k\ne2,\quad \theta_2=1/8,
\]
so that for all $k\ge1$,
\[
\mu_k=\frac{16}7\prod_{j\le k}(1-\theta_k).
\]
Let $\nu_k=1/\theta_k$, so that $\nu_k=p_k$ for $k\ne2$ and $\nu_2=8$.
With these new values, we continue to have
\eqref{small2primenew} recorded anew as follows:
\begin{equation}
 \label{small2primelcm}   
   \sum_{\substack{t\in T\\Q(t)<t^{\theta}}}\frac1{t^\mu}\ 
   ~=~\int_{p_{\nu+1}^{\nu}}^{\infty}\frac{\mu}{z^{1+\mu}}N(z)\,dz 
 ~\le ~ \frac{ \mu}{\mu - \frac{1}{k} - \theta} \,p_{\nu+1}^{-\nu (\mu - \frac{1}{k} - \theta)},
\end{equation}
where $\mu=\mu_k$, $\theta=\theta_k$, $\nu=\nu_k$.

We have the analogue of \eqref{large2prime},
where $\lambda$ is replaced with $\mu=\mu_k$ and $\lambda'$ is replaced with $\mu'=\mu_{k-1}$.  In addition, we continue to have Claim 1, checking that $p^{-\mu}\le\frac13$ for $p\ge p_k$. 

However, Claim 2 needs to be verified.  As before, we check that $S(p_k)<0$ for $2\le k\le 199$.  Following the argument for $k\ge200$, we have $1-2\mu>0.528$, $p_k^{\mu-\mu'}>1-1/p_k$, and $e^{-1.7}<p_k^{-\lambda}<e^{-1.65}$, again 
following from Lemma \ref{primeineq2}.  Thus,
\begin{align*}
    S(p_k)&<-\frac{1.65e^{-1.7}}{\log p_k}\Big(\Big(1-\frac1{p_k}\Big)\Big(1-\frac1{\log p_k}\Big)-\frac{2.03248}{0.528}e^{-1.65}\Big)\\
    &<-\frac{0.035}{\log p_k},\quad\hbox{ for }~k\ge200.
\end{align*}
This is somewhat stronger than Claim 2.

We have the analogue of \eqref{almostdone2}:
\begin{equation}
    \label{almostdone3}
 I_\mu\ < \ \frac{ \mu}{\mu - \frac{1}{k} - \theta}\, p_{\nu+1}^{-\nu(\mu - \frac{1}{k} - \theta)} 
 + S(Y),   
\end{equation}
where the first term does not occur if $Y<p_{\nu+1}$.  Our goal is to show that $I_\mu<0$.  Thus, by Claims 1 and 2, we may assume that $Y=p_{\nu+1}$. We then check numerically that the bound in \eqref{almostdone3} is negative for $2\le k\le 199$.

To show that $I_\mu<0$ for $k\ge200$, note that
\[
\mu-\frac1k-\theta>\frac{1.6}{\log p_k},\quad \frac{\mu}{\mu-\frac1k-\theta}<1.05
\]
in analogy to what we had before.  Thus,
\[
\frac{\mu}{\mu-\frac1k-\theta}\,p_{\nu+1}^{-\nu(\mu-\frac1k-\theta)}<1.05e^{-1.6p_k},
\]
which is smaller than before.  Hence $I_\mu<0$ for $k\ge200$, which completes the proof.

\section{Theorem for strongly $k$-primitive sets}
\label{strong}
In this section we prove Theorem \ref{indistinct}.

As in Section \ref{sec:main} we may assume that $A=T$ consists of composite numbers,  for each $p\in\mathcal P(T)$ we have $|T_p|\ge2$, and $\mathcal P(T)$ consists of all of the primes up to some point $Y$.
Note that since $\tau^{(\rm{s})}_k\le\tau_k$ for all $k$, Theorem \ref{indistinct} follows from Theorem \ref{thm:actuallymain} when $k\le 38$.  Thus,
in the sequel, we assume that $k\ge39$ and that the theorem holds for $k-1$.

For $k\ge39$, let
\[
\lambda=\lambda_k=\frac{3\log k}k,\quad \theta=\theta_k=1-\frac{\lambda_k}{\lambda_{k-1}}.
\]
A simple calculation shows that
\[
\nu=\nu_k:=\frac1{\theta_k}>\frac{k\log(k-1)}{\log(k-1)-1}>k.
\]

Recall that $P(t)$ denotes the largest prime factor of $t$.
Let 
\[
T_0=\{t\in T:P(t)<t^\theta\}.
\]
We now prove a version of  Lemma \ref{smooth2} dealing with $T_0$.

\begin{lem} \label{smooth1}
Let $N_0(z)$ denote the number of $t\in T_0$ with
$t\le z$.  Then
\begin{align*}
N_0(z) \ \le \ z^{\frac{1}{k} + \theta}.
\end{align*}
\end{lem}
\begin{proof}
Let $t\in T_0,\,t\le z$.  If $t\le z^{1/k}$, let
$m_1(t)=t$.  Otherwise, say the prime factorization of $t$ is $p_1p_2\cdots p_r$, where $p_1\ge p_2\ge\dots\ge p_r$.  Let $j$
be minimal with $p_1\cdots p_j>z^{1/k}$.
Since all of these primes are $<t^\theta\le z^\theta$, we have $m_1(t):=p_1\cdots p_j\le z^{\frac1k+\theta}$.
Continuing in this fashion we obtain a factorization 
\[
t=p_1p_2\cdots p_r=m_1(t)m_2(t)\cdots m_{l_t}(t),\quad l_t\le k,~\hbox{ each }m_i(t)\le z^{\frac1k+\theta}.
\]
We claim that each $t$ has at least one factor $m_i(t)$ that does not appear in the analogous factorization for any other $t'\in T_0$.
Indeed, if each $m_i(t)=m_{j_i}(t'_i)$ for
some $t'_i\in T_0\setminus\{t\}$ with $j_i\le l_{t'_i}$, then $t\mid t'_1t'_2\cdots t'_{l_t}$, contradicting $T_0$ as strongly $k$-primitive.
By mapping $t$ to such a unique factor $m_i(t)$ we obtain a 1 to 1 function from $T_0$ to the integers in $(1,z^{\frac1k+\theta}]$, so proving the lemma.
\end{proof}

Because of the change in the definition of $N(z)$ we do not have \eqref{small2primenew}.  Instead, we argue as follows.
Note that every member of $T_0$ has at least $\lceil\nu\rceil$ prime factors, counted with multiplicity.
Thus, the least element of $T_0$ is at least $2^{\nu}$.  In addition, the second smallest member of $T_0$ must be $\ge3^{\nu}$.
Indeed, if there are two members smaller than this, then $P(t)<t^\theta$ implies they are both powers of 2, and hence $T_0$ is not primitive.
More generally, using Lemma \ref{Ysmall}, $T_0$ has at most $j$ members smaller than $p_{j+1}^{\nu_k}$ for each $j\le k$.  
Thus,
\begin{align} \label{smallprime}
\mathop{\sum_{t \in T_0}} \frac{1}{t^\lambda} 
&< \ \sum_{j\le k}\frac{1}{p_{j}^{\nu\lambda}} \ + 
 \int_{p_{k+1}^{\nu}}^{\infty}\frac{\lambda}{z^{1+\lambda}}N_0(z)\,dz \nonumber\\
& \ < \ \sum_{j\le k}\frac{1}{p_{j}^{\nu\lambda}} \ + 
\frac{ \lambda}{\lambda - \frac{1}{k} - \theta} \,p_{k+1}^{-\nu (\lambda - \frac{1}{k} - \theta)}
\end{align}
by partial summation and Lemma \ref{smooth1}.

In the next lemma we give a variant of Lemma \ref{derived2} in a more general setting.
\begin{lem}
\label{derived1}
Let $k\ge2$ and let $T$ be an arbitrary strongly $k$-primitive set of composite numbers such that for each prime $p\in\mathcal P(T)$, $|T_p|\ge2$.
Then the map $f:T\to\N$ given by 
$f(t) = t/P(t)$ is $1$ to $1$ and $f(T)$ is $(k-1)$-primitive.
\end{lem}
\begin{proof}
Suppose $t,t'\in T$, $t\ne t'$, and $f(t)=f(t')$.  Since $|T_{P(t)}|\ge2$, there is some $s\in T\setminus\{t\}$ with $P(t)\mid s$.  Then
\[
t=P(t)\cdot\frac{t}{P(t)}=P(t)\cdot\frac{t'}{P(t')}\ {\Big|}\ st',
\]
contradicting $T$ as strongly 2-primitive.  Thus, $f$ is 1 to 1.  

Next, suppose that 
 $f(T)$ is not strongly $(k-1)$-primitive, so that there are $t,t_1,\dots t_{k-1}$ in $T$ with $t\notin\{t_1,\dots,t_k\}$ and
$$
f(t)\mid f(t_1)\cdots f(t_{k-1}).
$$
With $P(t)\mid s\ne t$ as above, we have
$t\mid s\cdot t_1\cdots t_{k-1}$, contradicting $T$ as strongly $k$-primitive.
Thus, $f(T)$ is strongly $(k-1)$-primitive, and the proof is complete.
\end{proof}

Let $T_\theta=T\setminus T_0=\{t\in T:P(t)\ge t^\theta\}$.  We apply Lemma \ref{derived1} to $T$, and so restricting the injection $f$ to $T_\theta$, we have $f(T_\theta)$ as a $(k-1)$-primitive set.
Further,
every $t\in T_\theta$ has $f(t)\le t^{1-\theta}$.  Thus,
 $t^{-\lambda} \le (t/P(t))^{- \lambda / (1 - \theta)}$ and by the induction hypothesis on the $(k-1)$-primitive set $f(T_\theta)$,
\begin{align*}
\sum_{t\in T_\theta}t^{-\lambda} \ \le \ \sum_{t\in T_\theta}\big(t/P(t)\big)^{-\lambda'} \ = \ \sum_{d\in f(T'_\theta)} d^{-\lambda'} \ \le \  \sum_{p \le Y} p^{-\lambda'}
\end{align*}
for $\lambda' :=\lambda_{k-1} = \frac{\lambda_k}{1-\theta}$. 

By way of \eqref{smallprime}, this allows us to replace \eqref{almostdone2} with
\[
I_\lambda=\sum_{t\in T}t^{-\lambda}-\sum_{p\le Y}p^{-\lambda}<
\sum_{p\le p_k}p^{-\nu\lambda}+\frac{\lambda}{\lambda-\frac1k-\theta}\,p_{k+1}^{-\nu(\lambda-\frac1k-\theta)}+\sum_{p\le Y}\big(p^{-\lambda'}-p^{-\lambda}\big),
\]
with the goal as before to show that $I_\lambda<0$.

By the mean value theorem, there is some $\xi\in(\lambda,\lambda')$ with 
\[
\sum_{p\le Y}\big(p^{-\lambda'}-p^{-\lambda}\big)= 
-(\lambda'-\lambda)\sum_{p\le Y}\frac{\log p}{p^\xi}<-\lambda'\theta\sum_{p\le Y}\frac{\log p}{p^{\lambda'}}.
\]
Since by Lemma \ref{Ysmall} we may assume that $Y\ge p_{k+1}$, it suffices, by Lemma \ref{primeineq2}, for us to show that
\begin{equation}
    \label{goal2}
\sum_{p\le p_k}p^{-\nu\lambda}+\frac{\lambda}{\lambda-\frac1k-\theta}\,p_{k+1}^{-\nu(\lambda-\frac1k-\theta)} 
<\lambda'\theta p_{k+1}^{1-\lambda'}\Big(1-\frac1{\log p_{k+1}}\Big).
\end{equation}
Now $\nu\lambda>3\log k$, so that
\[
\sum_{p\le p_k}p^{-\nu\lambda}<2^{-3\log k}+(k-1)3^{-3\log k}<k^{-2}+k\cdot k^{-3}=2k^{-2}.
\]
Using $k\ge39$ we see that $\nu(\lambda-\frac1k-\theta)>3\log k-2$
and $\lambda/(\lambda-\frac1k-\theta)<1.23$, so that
\[
\frac{\lambda}{\lambda-\frac1k-\theta}\,p_{k+1}^{-\nu(\lambda-\frac1k-\theta)} 
<1.23p_{k+1}^{-(3\log k-2)}<k^{-2}.
\]
So the left side of \eqref{goal2} is $<3k^{-2}$.  We now get a lower bound for the right side.
Using $k\ge39$, we have $\lambda'\theta>2(\log k)/k^2$ and
$p_{k+1}^{\lambda'}<4.4$.
Thus, 
\[
\lambda'\theta p_{k+1}^{1-\lambda}\Big(1-\frac1{\log p_{k+1}}\Big)
>0.79\frac{2\log k}{k^2}p_{k+1}/4.4
>\frac{0.36p_{k+1}\log k}{k^2}
>\frac{0.36\log^2k}{k},
\]
using that $p_{k+1}>p_k>k\log k$.  We do indeed have $3/k^2<0.36(\log^2k)/k$ when
$k\ge 39$, so we have \eqref{goal2}, and the theorem.

\section*{Acknowledgments}
The authors would like to acknowledge the University of Memphis for hosting the 2019 Erd\H{o}s Lecture Series, during which the initial ideas for this paper and its prequel were conceived. The second named author is supported by a Clarendon Scholarship at the University of Oxford.

\bibliographystyle{amsplain}

\begin{thebibliography}{9}

\bibitem{BM} W. D. Banks and G. Martin, Optimal primitive sets with restricted primes, {\it Integers} {\bf 13} (2013), \#A69, 10 pp.
    
\bibitem{C11} T. H. Chan, On sets of integers, none of which divides the product of $k$ others, {\it European J. Comb.} {\bf 32} (2011), 443--447.
    
\bibitem{CGS} T. H. Chan, E. Gy\H{o}ri, and A. S\'{a}rk\"{o}zy, On a problem of Erd\H{o}s on integers, none of which divides the product of $k$ others, {\it European J. Comb.} {\bf 31} (2010), 260--269.

\bibitem{CLP} T. H. Chan, J. D. Lichtman, and C. Pomerance, A generalization of primitive sets and a conjecture of Erd\H{o}s,  {\it Discrete Analysis} {\bf 16} (2020), 13 pp.

\bibitem{E35} P. Erd\H{o}s, Note on sequences of integers no one of which is divisible by any other, {\it J. London Math. Soc.} {\bf 10} (1935), 126--128.


\bibitem{E38} P. Erd\H{o}s, On sequences of integers no one of which divides the product of two others and on some related problems, {\it Tomsk. Gos. Univ. Ucen. Zap.} {\bf 2} (1938), 74--82.

\bibitem{E48} P. Erd\H{o}s, On the integers having exactly $k$ prime factors, {\it Ann.\ of Math.} {\bf 49} (1948), 53--66.

\bibitem{ESS} P. Erd\H{o}s, A. S\'{a}rk\"{o}zy, and E. Szemer\'{e}di, On an extremal problem concerning primitive sequences, {J. London Math. Soc.} {\bf 42} (1967), 484--488.


\bibitem{LAlmost} J. D. Lichtman, \textit{Almost primes and the Banks--Martin conjecture}, Journal of Number Theory, {\bf 211} (2020), 513--529.

\bibitem{LP} J. D. Lichtman and C. Pomerance, The Erd\H{o}s conjecture for primitive sets, {\it Proc. Amer. Math. Soc.}, Series B, {\bf 6} (2019), 1--14.
 
\bibitem{PS} P. P. Pach and Cs. S\'{a}ndor, Multiplicative bases and an Erd\H{o}s problem, {\it Combinatorica}, {\bf 38} (2018), no. 5, 1175--1203.

\bibitem{RS}  J. B. Rosser and L. Schoenfeld, Approximate formulas for some functions of prime numbers, {\it Illinois J. Math.} {\bf 6} (1962), 64--94.

\end{thebibliography}

\end{document}